\documentclass[12pt, twoside, reqno]{amsart}

\usepackage{amsmath, amsthm, amssymb, color, graphics}
\usepackage[english]{babel}
\usepackage{amsfonts}
\usepackage{bbm}
\usepackage{verbatim}
\usepackage{microtype}

\usepackage[headsep=30pt,footskip=25pt, left = 1.2in,  right = 1.2in, bottom=1.3in,
top=1.3in]{geometry}

\usepackage[utf8]{inputenc} 
\usepackage{mathrsfs} 
\usepackage[autostyle]{csquotes} 
\usepackage{tikz-cd} 
\usepackage{mathtools} 
\usepackage{cite} 
\usepackage{colonequals} 
\usepackage[linktocpage]{hyperref}

\newcommand{\Z}{\mathbb{Z}}
\newcommand{\N}{\mathbb{N}}

\def\rar{\rightarrow}
\def\Rar{\Rightarrow}

\def\CO{{\mathcal{O}}}

\theoremstyle{plain}
\newtheorem{theorem}{Theorem}[section]
\newtheorem{proposition}[theorem]{Proposition}
\newtheorem{lemma}[theorem]{Lemma}

\newtheorem{definition}[theorem]{Definition}

\theoremstyle{remark}
\newtheorem{remark}[theorem]{Remark}

\newcommand{\aisle}[1]{\ensuremath{\mathcal{#1} ^{\leq 0}}}

\newcommand{\caisle}[1]{\ensuremath{{\sf{Aisle ^{cp} _{\otimes}}}(#1)}}

\newcommand{\cat}[1]{\ensuremath{\mathcal{#1}}}

\newcommand{\coaisle}[1]{\ensuremath{\mathcal{#1} ^{\geq 0}}}

\newcommand{\derive}[1]{\ensuremath{\mathbf{D}_{\text{qc}}(#1)}}

\newcommand{\dperp}[1]{\ensuremath{^{\perp}(\mathcal{#1}^{\perp})}}

\newcommand{\dperpn}[1]{\ensuremath{^{\perp}(\mathcal{#1}[\mathbb N]^{\perp})}}

\newcommand{\fil}[1]{\ensuremath{{\sf{Thom}^{fil}}(#1)}}

\newcommand{\gsupp}[1]{\ensuremath{\varphi _{#1}}}

\newcommand{\Hom}{\ensuremath{\mathrm{Hom}}}

\newcommand{\perfect}[1]{\ensuremath {\mathrm{{Perf}}{(#1)}}}

\newcommand{\perfn}[1]{\ensuremath {\mathrm{{Perf}}^{\leq N}{(#1)}}}
 
\newcommand{\pre}[1]{\ensuremath{^{\leq #1}}}

\newcommand{\preaisle}[1]{\ensuremath{\langle \mathcal{#1} \rangle ^{\leq 0} }}

\newcommand{\prim}[1]{\ensuremath{{#1}^{\prime}}}

\newcommand{\scr}[1]{\ensuremath{\mathscr{#1}}}

\newcommand{\spec}[1]{\ensuremath{\mathrm{Spec } #1}} 

\newcommand{\Supph}{\ensuremath{\mathrm{Supph}}}

\newcommand{\Supp}{\ensuremath{\mathrm{Supp}}}

\newcommand{\taisle}[1]{\ensuremath{\langle \mathcal{#1} \rangle ^{\leq 0}_{\otimes}}}

\newcommand{\tensor}{\ensuremath{\otimes}-}

\title[Compactly generated tensor t-structures]{Compactly generated tensor t-structures on 
the derived category of a Noetherian scheme}
\date{} 
\dedicatory{}

\author[Umesh V Dubey]{Umesh V Dubey}
\address{Harish-Chandra Research Institute,  A CI of Homi Bhabha National Institute,  Chhatnag 
Road,  Jhunsi,  Prayagraj  211019,  India.}
\curraddr{}
\email{umeshdubey@hri.res.in}
\thanks{}

\author[Gopinath Sahoo]{Gopinath Sahoo}
\address{Harish-Chandra Research Institute,  A CI of Homi Bhabha National Institute,  Chhatnag 
Road,  Jhunsi,  Prayagraj  211019,  India.}
\curraddr{}
\email{gopinathsahoo@hri.res.in}
\email{gsahoo@math.tifr.res.in}
\thanks{}

\begin{document}

\begin{abstract}

	We introduce a tensor compatibility condition for t-structures.  For any Noetherian scheme
	$X$, we prove that there is a one-to-one correspondence between the set of 
	Thomason filtrations and the set of aisles of compactly generated tensor compatible 
	t-structures on the derived category of $X$.  This generalizes the earlier classification of 
	compactly generated t-structures for commutative rings to schemes.  Hrbek and Nakamura 
	have reformulated the famous telescope conjecture for t-structures.  As an application of our 
	main theorem, we prove that a tensor version of the conjecture is true for separated 
	Noetherian schemes.  
	
\end{abstract}

\maketitle

\tableofcontents



\section*{Introduction}

The classification of subcategories in terms of subsets of the ambient space, in the setting of 
derived categories, first appeared in the work of Hopkins.  Inspired by
the analogous result in the finite stable homotopy category \cite[Theorem 7]{HS98},  Hopkins 
obtained that for a 
commutative ring $R$ there is a one-to-one correspondence between thick subcategories of 
perfect complexes on $R$ and the specialization closed subsets of $\spec R$ 
\cite[Theorem 11]
{Hop87}.   Neeman intrigued by this parallel between the stable homotopy category and the 
derived category conducted further investigation.  He obtained the classification of localizing 
subcategories of the unbounded derived category $\mathbf{D}(R)$,  which on restricting to the 
perfect complexes $\perfect R$ provides Hopkins' theorem \cite[Theorem 1.5]{Nee92}.  
However,  Neeman showed all this is true for Noetherian rings and that extra care is needed for 
arbitrary rings; see \cite[Example 4.1]{Nee92}.

Thomason took a step further and generalized the results of Hopkins and Neeman to schemes.
For similar classification to hold in the case of schemes he observed that a tensor compatible 
condition on thick subcategories of $\perfect X$ has to be considered; a condition which all 
thick subcategories satisfy in the affine situation.  His second important observation was that 
for the non-Noetherian case specialization closed subsets were to be replaced by subsets 
stronger than specialization closed; such subsets are now being called Thomason subsets
after him; see Definition \ref{D Thomason subset}.  Thus he obtained for a quasi-compact 
quasi-separated scheme $X$,  a bijective 
correspondence between thick \tensor ideals of $\perfect X$ and Thomason subsets of $X$ 
\cite[Theorem 3.15]{Thomason},  thereby,  closing the gap pointed out by Neeman and 
establishing the correct bijective correspondence for arbitrary commutative rings; see his 
commentaries \cite[History 3.17]{Thomason}.  Similar classifications in the situation of 
Noetherian formal schemes \cite[Theorem 4.12]{AJS04} and for Noetherian graded rings 
\cite[Theorem 5.8]{AS13} have been achieved.  

All these works consider suitable triangulated 
subcategories in the derived category or perfect complexes. Another important class of 
subcategories, not necessarily triangulated,  is the class 
of t-structures  introduced in\cite{BBD} by Be\u{\i}linson et al. in their study of 
perverse sheaves.  Stanley showed the class of t-structures on $\mathbf{D}(\Z)$ is not a 
set \cite[Corollary 8.4]{Sta10} and therefore it is not feasible to classify all t-structures of the 
unbounded derived category in terms of subsets of $\spec R$.  However,  a subclass of 
t-structures on $\mathbf{D}(R)$, precisely the compactly generated ones, has been classified in 
terms of filtrations of specialization closed subsets by Alonso et al. \cite[Theorem 3.11]{AJS10},  
for $R$ Noetherian.  Hrbek has extended this classification to arbitrary commutative rings,  
\cite[Theorem 5.1]{Hrbek}.  Although he obtained a one-to-one bijection between the set of 
compactly generated t-structures and the set of  Thomason filtrations,  unlike the 
Noetherian case,  it is not clear that the aisle of such a t-structure is determined by 
cohomological supports; see the introduction to section 5 of \cite{Hrbek}.

In this article, we generalize the classification of compactly generated t-structures to 
Noetherian schemes. First, we introduce the notion of tensor t-structures; an almost similar 
notion has been studied in \cite[\S 5]{AJS03} with a different goal in mind.  Such t-structures 
localize well to the open subsets.  For a Noetherian scheme $X$ we prove:

\begin{theorem}
	There is a bijective correspondence between the set of aisles of compactly generated
	tensor t-structures on \derive X and the set of Thomason filtrations of $X$. 
	{\emph{(For more details see Theorem \ref{Theorem 1}.)}}
\end{theorem}

Though our approach is close to Thomason's proof  of \cite[Theorem 3.15]{Thomason} in spirit, 
our proof differs in two aspects.  Since the aisle of a tensor t-structure is rarely a \tensor ideal,  
one of the key ingredients of \cite{Thomason} - the Tensor Nilpotence Theorem or any naive 
variation of it,  is not helpful in our case.  Instead, we use a local global principle to obtain
Lemma \ref{Lemma 2} the counterpart of \cite[Lemma 3.14]{Thomason}.  Another 
obstacle was extending perfect complexes on an open subscheme to perfect complexes on 
the ambient scheme along an aisle.  Thomason and Trobaugh in their seminal 
paper \cite{TT90} give a criterion for the extension of perfect complexes using $K_0$ groups; 
see \cite[Lemma 5.6.2]{TT90}.  In the case of t-structures, one needs to consider the supports 
of cohomology sheaves component-wise,  instead of their union as in the stable case.  Thus,
we needed to have extensions of perfect complexes with some boundedness 
assumption.  This we achieve after closely inspecting and later modifying various results of 
\cite[\S 5]{TT90}; see section \ref{Ext}.

A Bousfield subcategory of a triangulated category is smashing if the corresponding 
localization functor preserves coproducts.  The telescope conjecture asks if all such smashing 
subcategories are compactly generated.  The question was originally asked by Ravenel for the 
stable homotopy category of spectra \cite{Ravenel},  in this case, the answer remains elusive.  
However,  in the 
case of algebraic triangulated categories, there are many positive results.  The first one is due 
to Neeman\cite[Corollary 3.4]{Nee92}. For a detailed discussion on the history of the 
conjecture and various positive results in this direction see the introduction of \cite{Antieau}.  
Since Bousfield subcategories correspond to stable t-structures, it is natural to seek a 
reformulation of the conjecture which encompasses t-structures.  Hrbek and Nakamura have 
formulated the telescope conjecture for homotopically smashing t-structures in the language of 
derivators; see \cite[Question A.7]{HN21}.  They proved for a commutative Noetherian ring $R$ 
the telescope conjecture for homotopically smashing t-structures is true,  more precisely,  any 
homotopically smashing t-structure on $\mathbf{D}(R)$  is compactly generated 
\cite[Theorem 1.1]{HN21}.  As an application of Theorem \ref{Theorem 1}, we obtain, for a 
separated Noetherian scheme $X$:
\begin{theorem} 
	Any homotopically smashing tensor t-structure on $\mathbf{D}(\mathrm{Qcoh}(X))$ is 
	compactly generated.  {\emph{(For more details see \S \ref{section 5})}}
\end{theorem}

This provides yet another proof of the tensor telescope conjecture in this case; see Remark
\ref{remark 3} and also \cite[Corollary 6.8]{BF11}.

In section \ref{section 1} we discuss the basic facts about t-structures, we recall that 
a t-structure is completely determined by its aisle.  The notion of compactly generated 
t-structure appears in various guises.  Though all these notions are known to be 
equivalent,  a proof seems to be missing in the literature.  Hence we give a complete treatment
of such t-structures in this section.  In section \ref{section 2} we introduce tensor t-structures 
and \tensor aisles of a tensor triangulated category and collect some basic facts about them.  
We show the relation between the new notion of \tensor aisle and the earlier notion of \tensor 
ideal; see Proposition \ref{tensor ideal}.  In the third section,  for a Noetherian scheme $X$, we 
prove that the associated subcategory of a
Thomason filtration is a compactly generated \tensor aisle of \derive X; see Theorem 
\ref{Compact generation}.  In section \ref{section 4}, we prove the main classification result.
In the last section,  we discuss the telescope conjecture for homotopically smashing 
t-structures and prove that a tensor version of the conjecture is true for separated Noetherian
schemes.

\section{Preliminaries}
\label{section 1}
 
 \subsection*{Convention} We always assume a triangulated category \cat T has all 
 small coproducts unless otherwise stated. 

\subsection{Basics on t-structures}

Let $\cat T$ be a triangulated category.  The notion of t-structure
is introduced by Be\u{\i}linson,  Bernstein, and Deligne in \cite[D\'{e}finition 1.3.1]{BBD}. 

\begin{definition}
A \emph{t-structure} on $\cat T$  is a pair of full subcategories $(\aisle T , \coaisle T)$ 
satisfying the following conditions :
\begin{itemize}
	\item[(t1)] For $A \in \aisle T $ and $B \in \coaisle T [-1]$,  $\Hom (A, B)  = 0$;
	\item[(t2)] $\aisle T [1] \subset \aisle T $ and $\coaisle T [-1] \subset \coaisle T $;
	\item[(t3)] For any $T \in \cat T$,  there is a distinguished triangle 
	\begin{center}
		$A \rar T \rar B \rar A[1]$
	\end{center}	 
	 such that $A \in \aisle T $ and $B \in \coaisle T [-1]$. 
\end{itemize}
\end{definition}

The triangle in (t3) is unique and we call it the t-decomposition triangle of $T$ for
(\aisle T, \coaisle T).  The full subcategory \cat H = \aisle T $\cap$ \coaisle T of \cat T is called 
the \emph{heart} of the t-structure (\aisle T, \coaisle T).  It is an abelian category and there is a 
natural cohomological functor $H^0 : \cat T \rar \cat H$.

In \cite{KV88} Keller and Vossieck observed that the pair (\aisle T, \coaisle T) is completely
determined by the subcategory \aisle T.  They characterized such subcategories and termed 
them as \emph{aisles} of \cat T.  We recall the precise definition of aisle \cite[1.1]{KV88}.

\begin{definition}
A full subcategory \cat U is an \emph{aisle} of \cat T if it satisfies the following conditions:
\begin{itemize}
	\item[(a1)] closed under positive shifts,  i.e.  $ \cat U [1] \subset \cat U$;
	\item[(a2)] closed under extensions, i.e.  for each distinguished triangle 
	\[ A \to B \to C \to A[1]\] in \cat T,
	 if $A$ and $C$ are in \cat U then $B$ is in \cat U;
	\item[(a3)] the inclusion $\cat U \rar \cat T $ admits a right adjoint denoted by $\tau _{\cat 
	U} ^{\leq}$. 

\end{itemize}

The functor $\tau _{\cat U} ^{\leq}$  is called the truncation functor associated with \cat U.  
\end{definition}

For any subcategory \cat U of \cat T, we denote $\cat U ^{\perp}$ to be the full subcategory 
consisting of objects $B \in \cat T$ such that $\Hom (A, B) = 0$ for all $A \in \cat U$.  
Analogously we define $^{\perp} \cat U $ to be the full subcategory of objects  
$B \in \cat T$ such that $\Hom (B, A) = 0$ for all $A \in \cat U$. 

The assignments, (\aisle T, \coaisle T) $\mapsto$ \aisle T and 
\cat U $\mapsto$ (\cat U,  \cat U$^{\bot}$[1]) give a mutually inverse bijective correspondence 
between the aisles of \cat T and the t-structures on \cat T; see \cite[\S 1]{KV88}.   The 
subcategory $ (\aisle T) ^{\perp} = \coaisle T[-1]$ is called the coaisle of the t-structure 
(\aisle T,  \coaisle T).

We define a few notions weaker than aisle; these have been introduced in \cite[\S 1]{AJS03}.

\begin{definition}
 A full subcategory \cat U is a \emph{preaisle} of \cat T if \cat U is closed under positive 
 shifts and extensions,  that is,  \cat U satisfies (a1) and (a2).  We call a preaisle \emph{stable}
 if it is closed under negative shifts.  A preaisle is a 
 \emph{thick preaisle} if it is closed under taking direct summands in \cat T.  A preaisle
 that is closed under taking small coproducts in \cat T is called a \emph{cocomplete preaisle}.  
 \end{definition}
 
 A cocomplete preaisle is thick by the Eilenberg swindle argument; 
 see \cite[Corollary 1.4.]{AJS03}.  By \cite[Lemma 1.3.]{AJS03} aisles are cocomplete preaisles.

\subsection{Compactly generated t-structures} 

A triangulated subcategory $\cat C$ of \cat T is \emph{thick} if it is closed 
under direct summands.  A thick subcategory is \emph{localizing} if it is closed under taking 
coproducts in \cat T.  If \cat S $\subset$ \cat T is a class of objects,  then 
$\langle \cat S \rangle$ denotes the smallest localizing subcategory of \cat T containing \cat S;
see for instance \cite[Definition 1.12]{Neeman}.  Following similar customs,  we denote the 
smallest cocomplete preaisle containing \cat S by \preaisle S and call it the \emph{cocomplete 
preaisle generated by \cat S.}

An object $C \in$ \cat T is \emph{compact} if for any collection of objects $\{A_i \}$
in $\cat T$ the natural map \[\coprod_i \Hom (C, A_i)\rar  \Hom (C,  \coprod_i A_i ) \] is an 
isomorphism. We denote the full subcategory of compact objects of $\cat T$ by $\cat T ^c$; 
the subcategory $\cat T^c$ is a thick subcategory of $\cat T$.  In contrast to our convention 
$\cat T^c$ need not have coproducts.

\begin{definition}
\label{definition}
 A preaisle \cat U is \emph{compactly generated} if  \cat U = \preaisle S for a set of compact 
 objects \cat S.   We say an aisle is \emph{compactly generated} if it is a compactly generated 
 preaisle.  A t-structure is \emph{compactly generated} if the aisle of the 
 t-structure is compactly generated.
\end{definition}

We state a result of Keller and Nicol\'{a}s.
 
\begin{proposition}[\text{\cite[Theorem A.7]{KP13}}]
\label{keller}
Let  \cat S be a set of compact objects of \cat T.  Then,

	\begin{itemize}
	\item[(1)] \preaisle S is an aisle of \cat T.
	
	\item[(2)] Every object $A$ of \preaisle S fits in a triangle
	\[ \coprod_{i \geq 0} A_i \rar A \rar \coprod_{i \geq 0} A_i[1] \rar \coprod_{i \geq 0} A_i[1] \]
	where $A_i$ is an $i$-fold extension of small coproducts of non negative shifts 
	of objects of \cat S.
	\end{itemize}

\end{proposition}

Now we bring another auxiliary notion from \cite[1.1,  page. 316]{AJS10} which will be helpful in
characterizing compactly generated preaisles.  Let \cat U be a preaisle of \cat T.  We say \cat U 
is a \emph{total} preaisle of \cat T if \dperp U $= \cat U$.  Next, we include some results which
are probably well known but have not appeared explicitly in the literature.

\begin{lemma}
\label{aisle total}
	Aisles are total preaisles. 
\end{lemma}

\begin{proof}
	\cite[Proposition 1.1 (i)]{AJS03}.
\end{proof}

Let \cat S be a class of objects of \cat T.    Consider the collection 
of objects $ \{ S[n] \mid S \in \cat S \text{ and } n \in \mathbb N \cup \{0\} \}$ we denote it by
\cat S$[\mathbb N]$.

\begin{lemma}
\label{total}
	The subcategory \dperpn S is a cocomplete total preaisle of \cat T 
\end{lemma}

\begin{proof}
	For any collection of objects $\cat Y \subset \cat T$, we have $\cat Y \subset ((^{\perp} \cat 
	Y)^{\perp})$,  replacing \cat Y by $(\cat Y^{\perp})$ we get $(\cat Y^{\perp}) \subset 
	((^{\perp} (\cat Y^{\perp}))^{\perp})$.  Applying left $^\perp$ we get $ ^{\perp} 
	((^{\perp} (\cat Y^{\perp}))^{\perp}) \subset  (^{\perp}(\cat Y^{\perp}))$.  
	Again for any collection $\cat Y$ we have $ \cat Y \subset$ $ ^\perp (\cat Y ^{\perp})$,
	replacing $\cat Y$ by $(^{\perp}(\cat Y^{\perp}))$ we get $ (^{\perp}(\cat Y^{\perp})) 
	\subset$ $ ^\perp ((^{\perp}(\cat Y^{\perp})) ^{\perp})$.
	This proves 
	\begin{center}
		$ (^{\perp}(\cat Y^{\perp})) =$ $ ^\perp ((^{\perp}(\cat Y^{\perp})) ^{\perp}).$
	\end{center}

	Replacing $\cat Y$ by the class of objects $\cat S[\N]$	we get
	$ (^{\perp}(\cat S[\N]^{\perp})) =$ $ ^\perp ((^{\perp}(\cat S[\N]^{\perp})) ^{\perp}).$ Also 
	note that for a class of objects \cat Y of \cat T,  if $\cat Y[-1] \subset \cat Y$ then 
	$^\perp \cat Y$ is a cocomplete preaisle.  Since $\cat S[\mathbb{N}]^{\perp} [-1] \subset 
	\cat S[\mathbb{N}]^{\perp}$ therefore \dperpn S is a  cocomplete total preaisle of \cat T. 
	
\end{proof}

\begin{lemma}
\label{total 2}
	The subcategory \dperpn S is the smallest total preaisle containing \cat S. 
\end{lemma}

\begin{proof}
	By Lemma \ref{total}, the subcategory \dperpn S is a total preaisle.  
	Suppose \cat U is a total preaisle containing \cat S.  As $\cat U$ is closed under positive 
	shifts $\cat S[\mathbb{N}] \subset \cat U$, hence $ \dperpn S \subset$ $ \dperp U = \cat U 
	$.   This proves the claim. 
\end{proof}

\begin{lemma}
\label{total 3}
	Let  \cat S be a set of compact objects of \cat T.  Then,
	\[ \preaisle S = { \dperpn S}.\]
\end{lemma}

\begin{proof}
	Recall that $\preaisle S$ is the smallest cocomplete preaisle containing $\cat S$.  Now
	$\dperpn S$ is a cocomplete preaisle by Lemma \ref{total},  hence \preaisle S $\subset$ 
	\dperpn S.   We know \preaisle S is an aisle by Proposition \ref{keller}.  Since aisles
	are total preaisles by Lemma \ref{aisle total}, \preaisle S is a total preaisle containing \cat S.   
	Lemma \ref{total 2} says \dperpn S is the smallest total preaisle containing \cat S 
	therefore we get  the reverse inclusion \dperpn S $\subset$ \preaisle S.
\end{proof}

\begin{proposition}
\label{compact characterization}
	Let  \cat U be a cocomplete preaisle of \cat T. 
	The following are equivalent:
	\begin{itemize}

		\item[(1)] The cocomplete preaisle \cat U is compactly generated.
		\item[(2)] There is a set of compact objects \cat S such that \cat U = \preaisle S.
		\item[(3)] There is a set of compact objects \cat S such that \cat U = \dperpn S.
		\item[(4)] There is a set of compact objects \cat S $\subset$ \cat U such that for any 
		$A \in \cat U$
		if \[\mathrm{Hom } (S ,  A) = 0\] for every $S \in \cat S[\mathbb N]$ then $A \cong 0$.
		\item[(5)] The pair $(\cat U,\cat U^{\bot}[1])$ is a compactly generated t-structure.
		
	\end{itemize}
\end{proposition}

\begin{proof}
	(1) $\Leftrightarrow$ (2) follows from Definition \ref{definition}.
	Lemma \ref{total 3} shows (2) $\Leftrightarrow$ (3).   Suppose $A \in \cat U$ and $\Hom (S, 
	A) = 0$ for all $S \in \cat S[\N]$, this means $A \in \cat S [\N] ^{\perp}$.  By Lemma 
	\cite[Lemma 3.1]{AJS03}  we have 
	$\cat S [\N] ^{\perp} = \cat U ^{\perp}$ hence $A \in \cat U \cap \cat U ^{\perp}$.  So
	we get $A \cong 0$; this proves (3) $\Rightarrow$ (4). 
	
	Next, we prove (4) $\Rightarrow$ (2).  Note that \preaisle S $\subset$ \cat U and 
	\preaisle S is an aisle by Proposition \ref{keller}.  Let $A \in \cat U$.
	Consider the t-decomposition triangle of $A$ for \preaisle S,
	\[\tau_{\cat S} ^{\leq} A \rar A \rar \tau_{\cat S}^> A \rar \tau_{\cat S} ^{\leq} A[1] .\]
	As \cat U is closed under extension $\tau_{\cat S}^> A$ is in \cat U.  Since $\tau_{\cat 
	S}^> A$ is in $\cat S[\mathbb N]^{\perp}$ by (4) $\tau_{\cat S}^> A \cong 0$.  Therefore
	the map $\tau_{\cat S} ^{\leq} A \rar A$ is an isomorphism and hence $A \in \preaisle S$.  
	This
	proves \cat U = \preaisle S.
	
	A compactly generated preaisle \cat U  is an aisle by Proposition \ref{keller}.  By
	the bijective correspondence between aisles and t-structures  and by Definition
	\ref{definition}, we get (5) $\Leftrightarrow$ (1). 
\end{proof}

\subsection{Smashing t-structures}

A thick subcategory $\cat C$ is \emph{Bousfield} if the inclusion $\cat C \to \cat T$ admits a 
right adjoint; see \cite[4.9]{Kra10}.   Therefore Bousfield subcategories are precisely the stable 
aisles.  A Bousfield subcategory \cat C is \emph{smashing} if the subcategory 
$\cat C ^{\perp} $ is localizing; for other characterizations see \cite[Proposition 5.5.1]{Kra10}.
In view of this we say an aisle \cat U is \emph{smashing} if $\cat U ^{\perp}$ is closed 
under coproducts; also see \cite[Definition 7.1]{SSV21}.  A t-structure $(\cat U,  \cat V[1])$ is 
called \emph{smashing} if the coaisle $\cat V$ is closed under taking coproducts.

\begin{proposition}
	Compactly generated t-structures are smashing t-structures.  
\end{proposition}

\begin{proof}
	Suppose $(\cat U, \cat V[1])$ is a compactly generated t-structure and $\cat U =$ 
	$\dperpn S$.
	Let $\{A _i\}$ be a collection of objects in $\cat V$.  To show $\coprod_i A_i \in \cat 
	V$ it is enough to show $\Hom (S , \coprod_i A_i) = 0$ for all $S \in \cat S[\N]$.  Since
	$S \in \cat S[\N]$ is compact we have $\Hom (S, \coprod_i A_i) = \coprod_i  \Hom (S , A_i).$
	As $A_i \in \cat V$ we have $\Hom (S, A_i) = 0$ and this proves the claim. 
\end{proof}

\section{Tensor t-structures}
\label{section 2}

We recall the definition of tensor triangulated category from \cite[Definition A.2.1]{HPS97}.

\begin{definition}
	A \emph{tensor triangulated category} $(\cat T,  \otimes , \mathbf{1})$ is a triangulated 
	category with a compatible closed symmetric monoidal structure.  This means there is a 
	functor $- \otimes - \colon \cat T \times \cat T \to \cat T$ which is triangulated in both the 
	variables and satisfies certain compatibility conditions.  Moreover,  for each $B \in \cat T$ the 
	functor $- \otimes B$ has a right adjoint which we denote by $\mathscr{H}om (B, -) $.   The 
	functor $\mathscr{H}om (-,-)$ is triangulated in both the variables, and for any $A$,  $B$,  
	and $C$ in $\cat T$ we have natural isomorphisms $ \Hom ( A \otimes B ,  C) \rar 
	\Hom (A,  \mathscr{H}om (B, C))$. 
\end{definition} 

Suppose $\cat T$ is given with a preaisle $\cat T ^{\leq 0}$ such that $\aisle T \otimes \aisle T
\subset \aisle T$ and $\mathbf{1} \in \aisle T$.  We introduce the following definition which is motivated from \cite[\S 5]{AJS03}; a similar situation is also considered in \cite[\S 3]{Nee18}.

\begin{definition}
\label{tensor definition}
A preaisle $\cat U$  of \cat T is a \emph{$\otimes$-preaisle} (with respect to \aisle T) if $\cat T ^{\leq 0} \otimes 
\cat U \subset \cat U$.  An aisle is called a \tensor aisle if it is a \tensor preaisle.  A t-structure is 
a \emph{tensor t-structure} if the aisle of the t-structure is a $\otimes$-aisle.
\end{definition}

\begin{remark}
	Though we define tensor t-structures in this generality,  for the derived categories (example 
	\derive X) we always take tensor t-structures with respect to the aisle of the 
	standard t-structure.
\end{remark}

\begin{proposition}
\label{tensor proposition}
	Let $(\cat U,  \cat V[1])$ be a t-structure on \cat T.  Then the following are equivalent:
	\begin{itemize}
		\item[(1)]  $(\cat U,  \cat V[1])$ is a tensor t-structure. 
		\item[(2)]  \cat U is a \tensor aisle.
		\item[(3)] $\mathscr{H}om ( \aisle T,  \cat V) \subset \cat V$.
	\end{itemize}
\end{proposition}

\begin{proof}
	(1)$\Leftrightarrow$(2) is by definition.   Let $A \in \cat U$, $B \in \cat V$, and $X \in 
	\aisle T$.  Then from the adjunction isomorphism we have
	\[ \Hom ( A \otimes X ,  B) \cong \Hom (A ,  \mathscr{H} om(X, B)).\]
	This proves (2)$\Leftrightarrow$(3).
	
\end{proof}

\begin{lemma}
\label{tensor generator 1}
	 Let \aisle {\cat T} be generated by a set of objects \cat K,  that is, \aisle {\cat T} =
	 \preaisle{\cat K}.  Then a preaisle \cat U of \cat T is a $\otimes$-preaisle if and only if 
	 \cat K $\otimes$ \cat U $\subset \cat U$.
\end{lemma}

\begin{proof}
	Suppose \cat K $\otimes$ \cat U $\subset \cat U$.  We define  $\cat B$ = $\{ X \in \aisle T 
	\mid X \otimes \cat U \subset \cat U\}$.  If $X \in 
	  \cat B$ then $X[1] \in \cat B$, as \cat U is closed under positive shifts.  Suppose we
	  have a triangle $X \rar Y \rar Z \rar X[1]$ with $X$ and $Z$ in \cat B.  Since \cat U is closed 
	  under extension,  we get $Y \in \cat B$.
	  This proves \cat B is a preaisle.  Again cocompleteness of $\cat U$ implies
	  $\cat B$ is cocomplete.  Since $\cat K \subset \cat B$ we get 
	  $\aisle T \subset \cat B$ which proves $\cat U$ is \tensor preaisle.  The converse follows 
	  easily.  
\end{proof}

\begin{proposition}
\label{affine aisle}
	If $\aisle T = \langle \mathbf{1} \rangle ^{\leq 0}$ then every cocomplete preaisle of \cat T is 
	a \tensor preaisle.  In particular, all t-structures are tensor t-structures.  
\end{proposition}

\begin{proof}
	Follows from Lemma \ref{tensor generator 1}.
\end{proof}

We recall the definitions of \emph{\tensor ideal} and \emph{coideal}; see for instance
\cite[Definition 1.4.3]{HPS97}.

\begin{definition}
	A thick subcategory $\cat C$ of \cat T is a \tensor ideal of $\cat T$ if $\cat T \otimes
	\cat C \subset \cat C$.  We say $\cat C$ is a \emph{coideal} of $\cat T$ 
	if $\mathscr{H}om(\cat T,  \cat C) \subset \cat C$.
	
\end{definition}

Recall that a preaisle is called stable if it is closed under negative shifts.  We say a t-structure is 
\emph{stable} if the aisle is closed under negative shifts.  Note that for a stable t-structure the 
coaisle is closed under positive shifts.  

\begin{lemma}
\label{tensor lemma 1}
	Suppose $ \langle \aisle T \rangle = \cat T$. If a cocomplete \tensor preaisle is stable, then it 
	is a  \tensor ideal of $\cat T$.  
\end{lemma}

\begin{proof}
	Let $\cat U$ be a stable cocomplete \tensor preaisle of $\cat T$.  We define  $\cat B$ = $\{ X 
	\in \cat T 
	\mid X \otimes \cat U \subset \cat U\}$.  Using similar arguments as in Lemma \ref{tensor 
	generator 1}, we can observe that $\cat B$ is a localizing subcategory of $\cat T$.   Since 
	$\cat B$ 
	contains 
	$\aisle T$ we get $\langle \aisle T \rangle = \cat T \subset \cat B$.  This proves $\cat U$ is a 
	\tensor ideal.
\end{proof}

\begin{proposition}
\label{tensor ideal}
	Let $\cat T$ be a tensor triangulated category with a preaisle $\aisle T$ such that
	$\langle \aisle T \rangle = \cat T$.  If $(\cat U,  \cat V [1])$ is a stable t-structure on 
	\cat T, then the following are equivalent:
	
	\begin{itemize}
		\item[(1)]  $(\cat U,  \cat V [1])$ is a stable tensor t-structure on \cat T.
		\item[(2)] The aisle \cat U is a \tensor ideal of \cat T.
		\item[(3)] The coaisle \cat V is a coideal of \cat T.
	\end{itemize}
\end{proposition}

\begin{proof}
	(1)$\Leftrightarrow$(2) follows from Lemma \ref{tensor lemma 1}.  (2)$\Leftrightarrow$(3)
	follows from the adjunction isomorphism - same as the proof of
	Proposition \ref{tensor proposition}.  
\end{proof}

Let $\cat T$ be a tensor triangulated category and \cat S be a class of objects of \cat T.  We 
denote the smallest cocomplete $\otimes$-preaisle containing \cat S  by \taisle S.  The 
following result is a version of \cite [Lemma 1.1]{HR17} in the present context.

\begin{lemma}
\label{key1}
Let $\cat T$ and $\cat T ^{\prime}$ be two tensor triangulated categories with preaisles 
$\aisle T$ and $\aisle {\prim{T}}$ respectively.  
Let $F : \cat T \rar \cat T ^{\prime}$ be a coproduct preserving tensor triangulated functor. 
We assume $F$ to be right-t-exact, that is, $F(\cat T ^{\leq 0}) \subset \cat T ^{\prime \leq 0}$.  
Then the following hold:
	\begin{itemize}
		\item[(1)] If $\cat U ^{\prime}$ is a cocomplete $\otimes$-preaisle of $\cat T ^{\prime}$, 
		then the full subcategory $F^{-1} (\cat U ^{\prime})$ of $\cat T$ is a cocomplete 
		$\otimes$-preaisle. 
		
		\item[(2)] If $\cat S \subset \cat T$ is a class of objects,  then $\langle F(\cat S) 
		\rangle^{\leq 0} _{\otimes} = \langle F( \langle \cat S \rangle^{\leq 0} _{\otimes} ) 
		\rangle^{\leq 0} _{\otimes}.$
	\end{itemize}

\end{lemma}

\begin{proof}

	Let $\cat U ^{\prime}$ be a cocomplete preaisle of $\cat T ^{\prime}$.  If $A \in F ^{-1}
	(\cat U ^{\prime})$,  then $F(A) \in \cat U ^{\prime}$ so $F(A)[1] \in \cat U ^{\prime}$ and 		
	this implies $A[1] \in F^{-1}(\cat U ^{\prime})$. Given a triangle $A \rar B \rar C \rar A[1]$ in $
	\cat T$ with $A$ and $C$ in $ F^{-1}(\cat U ^{\prime})$.  We have $F(A) \rar F(B) \rar F(C) 
	\rar F(A)[1]$ in $\cat T ^{\prime}$ with $F(A)$ and $F(C)$ in $\cat U ^{\prime}$ hence $F(B) 
	\in \cat U ^{\prime}$,  and we get $B \in F^{-1}(\cat U ^{\prime})$.  This proves $F^{-1}
	(\cat U ^{\prime})$ is a preaisle of $\cat T$.  The fact $F$ preserves coproducts will imply 
	that if $\cat U ^{\prime}$ is cocomplete then $F^{-1}(\cat U ^{\prime})$ is cocomplete.  Let 
	$\cat U ^{\prime}$ be a cocomplete $\otimes$-preaisle of $\cat T ^{\prime}$.  For any $A 
	\in F^{-1}(\cat U ^{\prime})$ and $B \in \cat T ^{\leq 0}$,  we have $F(A\otimes B) = F(A) 
	\otimes F(B) \in \cat U ^{\prime}$ and hence $A \otimes B \in F^{-1}(\cat U ^{\prime})$.  This 
	shows  $ F^{-1}(\cat U ^{\prime})$ is a cocomplete $\otimes$-preaisle.
	
	Now part (2).  For any class of objects $\cat S \subset \cat T$, we have
	 $\langle F(\cat S) \rangle ^{\leq 0} _\otimes \subset \langle F( \langle \cat S \rangle \pre 0_
	 \otimes) \rangle \pre 0 _\otimes$.   By (1), $ F ^{-1} ( \langle F( \cat S ) \rangle \pre 0 
	 _\otimes )$ is a cocomplete \tensor preaisle.  Since $\cat S \subset F ^{-1} ( \langle F( \cat S 
	 ) \rangle \pre 0 _\otimes )$  we have $ \langle \cat S \rangle \pre 0_\otimes
	 \subset F ^{-1} ( \langle F(\cat S )\rangle \pre 0_\otimes )$.  This implies $F(\langle \cat 
	 S \rangle \pre 0_\otimes ) \subset  \langle F(\cat S) \rangle \pre 0_\otimes $.  Therefore,  $
	 \langle F(\langle \cat S \rangle \pre 0_\otimes) \rangle \pre 0 _\otimes\subset
	 \langle F(\cat S) \rangle \pre 0_\otimes $.
	 
\end{proof}

Let $\cat T$ be a rigidly compactly generated tensor triangulated category in the sense of 
Balmer \cite[1.1]{BF11}.  In particular, $\mathbf{1} \in \cat T^c$ and the tensor product 
$\otimes$ restricts to $\cat T^c$.  In this case, we can prove a tensor triangulated analogue of 
Proposition \ref{keller}.

\begin{proposition}
\label{keller tensor}
	Let $\cat T$ be a rigidly compactly generated tensor triangulated category with a preaisle 
	$\aisle T$. 
	Suppose $ \aisle T$ is compactly generated and $\cat T^c \cap \aisle T = \cat K$.  
	Let \cat S be a set of compact objects of \cat T.  Then,
	
	\begin{itemize}
	\item[(1)] \taisle S is an aisle of 	\cat T.
	
	\item[(2)] Every object $A$ of \taisle S fits in a triangle
	\[ \coprod_{i \geq 0} A_i \rar A \rar \coprod_{i \geq 0} A_i[1] \rar \coprod_{i \geq 0} A_i[1] \]
	where $A_i$ is an $i$-fold extension of small coproducts of non negative shifts 
	of objects of $\cat K \otimes \cat S$.
	\end{itemize}
\end{proposition}

\begin{proof}

	If we show  \[\langle \cat K \otimes \cat S \rangle ^{\leq 0} = \taisle S\] then 
	by Proposition \ref{keller}, our claim follows.  Since $\cat K \otimes \cat S \subset 
	\taisle S$,  we have $\langle \cat K \otimes \cat S \rangle ^{\leq 0} \subset \taisle S$.  
	
	For convenience, we denote $\langle \cat K \otimes \cat S \rangle ^{\leq 0}$ by \cat U.  
	Let \cat A = $\{ X \in \cat U \mid \cat K \otimes X \subset \cat U \}$. 
	As in the proof of Lemma \ref{tensor generator 1}, it can be checked that \cat A is 
	a cocomplete preaisle.  Since $\cat K \otimes \cat K \subset \cat K$ we have 
	$\cat K \otimes \cat S \subset \cat A$,  and this proves $\cat A = \cat U$.  Which means
	$\cat K \otimes \cat U \subset \cat U$ now by Lemma \ref{tensor generator 1} we get 
	\cat U is a tensor preaisle.  Since $\mathbf{1} \in \cat K$ we have $\cat S \subset \cat U$,
	this shows $\taisle S \subset \cat U$.
	
\end{proof}

\section{A class of tensor t-structures in \derive X}
\label{section 3}

Let $X$ be a quasi-compact quasi-separated scheme.  We denote the derived category of  
complexes of $\CO _X$-modules with quasi-coherent cohomology by $\derive X$.
The derived category $(\derive X, \otimes ^{L} _{\CO _X},\CO_X)$ is a tensor triangulated 
category with the derived tensor product $\otimes ^{L} _{\CO _X}$  and the structure sheaf 
$\CO_X$ as the unit.  The full subcategory of complexes whose cohomologies vanish in 
positive degree $\mathbf{D}_{qc} ^{\leq 0}(X)$ is a preaisle of $\derive X$.  

\emph{We define the \tensor preaisles of $\derive X$ with respect to the standard preaisle
$\mathbf{D}_{qc} ^{\leq 0}(X)$.}

In the affine situation when $X = \spec R$, note that $R$ generates the  preaisle 
$\mathbf{D} ^{\leq 0} (R)$.  Hence by  Proposition \ref{affine aisle} every cocomplete
preaisle of $\mathbf{D}(R)$ is a \tensor preaisle.  This has been well known; see for instance 
\cite[Proposition 1.10]{AJS10} also \cite[Proposition 2.2]{Hrbek}.

The following is a characterization of \tensor preaisles in $\derive X$.  For separated schemes
with an ample family of line bundles (a.k.a.  divisorial),  see \cite[Proposition 5.1]{AJS03} for a 
stronger result; note that what we call \tensor aisle has been termed as \emph{0-rigid} in 
\cite[\S 5]{AJS03}.

\begin{lemma}
\label{tensor test}
	Let \cat U be a cocomplete preaisle of $\derive X$.  The following are equivalent:
	\begin{itemize}
	\item[(1)] The preaisle \cat U is a \tensor preaisle.
	\item[(2)]  For every $G \in \cat U$ and $F \in \mathbf{D}_{qc} 
	^{\leq 0}(X)$ we have 
	$F \otimes ^{L} _{\CO _X}G \in \cat U$.
	\item[(3)] For every $G \in \cat U$ and $F$ a quasi-coherent $\CO _X$-module,  we have 
	$F \otimes
	^{L} _{\CO _X} G \in \cat U$.
	\item[(4)] For every $G \in \cat U$ and $F$ a flat quasi-coherent $\CO _X$-module,  we have $F \otimes
	^{L} _{\CO _X} G \in 
	\cat U$.
	\end{itemize}
\end{lemma}

\begin{proof}

	(1)$\Leftrightarrow$(2) follows from the definition. The implications (2) $\Rightarrow$ (3) $
	\Rightarrow$ (4) are obvious. We prove (4) $\Rar$ (2).  Let $F \in \mathbf{D}_{qc} ^{\leq 0}
	(X)$,  we can replace $F$ by a K-flat resolution say $E$,  such that
	$E$ is in $ \mathbf{D}_{qc} ^{\leq 0}(X)$ and each component $E^{i}$ is  a flat $\CO _X$-
	module \cite[Proposition 2.5.5]{Lipman}.  For each integer $n$ we denote $\sigma ^{\geq n} 
	E$ the brutal truncation of $E$ from below.  Note that $E$ is the colimit of the directed 
	system $\{ \sigma^{\geq n} E \}$ with the obvious maps.
	
	For any $G \in \cat U$,  to show $E \otimes ^{L} _{\CO _X} G \in \cat U$ it is 
	enough to show for each $n \in \N$,  $\sigma ^{\geq -n} E \otimes ^{L} _{\CO _X} G \in
	 \cat U$ since \cat U is cocomplete.  We prove by induction on $n$.  We have the following 
	 distinguished triangle for each $n$,
	\[ \sigma^{\geq -(n -1)} E \rar \sigma ^{\geq -n} E \rar E^{-n}[n] \rar \sigma^{\geq -(n-1)} E[1] 
	.\]
	
	Here $E^{-n}[n]$ is the complex which is zero everywhere except at $-n$th entry, where it is 
	precisely $E^{-n}$: the $-n$th component of the complex E.  As $E^{-n}$ is a 
	flat $\CO _X$-module and \cat U is closed under positive shifts, by (4) we have  $E^{-n}[n] 
	\otimes ^{L} _{\CO _X} G \in \cat U$.  Now by induction hypothesis we assume $
	\sigma^{\geq -(n -1)} E \otimes ^{L} _{\CO _X} G \in \cat U$,  this proves the claim since $
	\cat U $ is closed 
	under extensions.
	
\end{proof}

\subsection{Filtrations of supports and associated subcategories}

Let $E$ be a complex in $\derive X$.  The \emph{cohomological support} of $E$ 
\cite[Definition 3.2]{Thomason} is defined to be the subspace  
\[ \Supph(E) = \bigcup _{n \in \Z} \Supp (H^n(E)). \] 

We introduce an auxiliary notation,  
 \[\Supph ^{\geq i} (E) = \bigcup_{j \geq i} \Supp( H^j(E)).\]

\begin{lemma}
\label{support lemma 1}
	Let $ A \rar B \rar C \rar A[1]$ be a distinguished triangle in $\derive X$. Then,
	\[\Supph^{\geq i}(B) \subset \Supph^{\geq i}(A) \bigcup \Supph^{\geq i}(C).\]
\end{lemma}

\begin{proof}
	From the long exact cohomology sequence we have
	\[\Supp (H^j(B)) \subset \Supp (H^j(A)) \bigcup \Supp (H^j(C)).\]
	
	Now taking union over $j \geq i$ the result follows.
\end{proof}

\begin{lemma}
\label{support lemma 2}
	For any $E \in \derive X$,  we have \[ \Supph^{\geq i}(E[1]) \subset \Supph^{\geq i}(E). \]
\end{lemma}

\begin{proof}
	Note that $ \Supph^{\geq i}(E[1]) = \Supph^{\geq i+1}(E)$.  And
	$\Supph^{\geq i+1}(E) \subset \Supph^{\geq i}(E)$ follows from the definition of $\Supph 
	^{\geq i}(-)$.
\end{proof}

\begin{lemma}
\label{support lemma 3}
	Let $\{ E_{\alpha}\}$ be a set of objects in $\derive X$.  Then,
	\[\Supph^{\geq i}(\oplus_{\alpha} E_{\alpha}) = \bigcup_{\alpha} \Supph^{\geq i}(E_{\alpha}).
	\]
\end{lemma}

\begin{proof}
	It is enough to show $\Supp ( H^i( \oplus_{\alpha} E_{\alpha})) =  \bigcup_{\alpha} \Supp 
	(H^{i}(E_{\alpha}))$.  As cohomology commutes with direct sums we have $H^i( 
	\oplus_{\alpha} E_{\alpha}) \cong \oplus_{\alpha} H^i(E_{\alpha})$ so
	
	\[\Supp (H^i( \oplus_{\alpha} E_{\alpha})) =
	\Supp ( \oplus_{\alpha} H^i(E_{\alpha})) =  \bigcup_{\alpha} \Supp (H^{i}(E_{\alpha})).\] 
\end{proof}

\begin{lemma}
\label{support lemma 4}
	Let $B \in \derive X$.  For any flat quasi-coherent $\CO _X$-module $F$ and $n \geq 0$, we 
	have
	\[\Supph ^{\geq i} ( F [n] \otimes^{\textbf{L}}_{\CO_X} B) \subset \Supph ^{\geq i}(B).\]
\end{lemma}

\begin{proof}

	First we prove for $n = 0$.  Since $F$ is a flat $\CO _X$-module we have $F \otimes 
	^{\textbf{L}} _{\CO _X} B =
	F \otimes _{\CO _X} B$.  The flatness of $F$ implies $H^i(F \otimes  _{\CO _X} B) 
	\cong F \otimes _{\CO _X} H^i(B)
	$.  Now $\Supp (F \otimes _{\CO _X} H^i(B)) \subset \Supp(H^i(B)),$ hence we get 
	\[ \Supp (H^{i} ( F \otimes^{\textbf{L}}_{\CO_X} B)) \subset \Supp (H ^{ i} (B)). \]  And this 
	proves 
	\[ \Supph ^{\geq i} ( F \otimes^{\textbf{L}}_{\CO_X} B) \subset \Supph ^{\geq i} (B). \]
	
	Next for $n > 1$.  By Lemma \ref{support lemma 2},  we have
	\begin{align*}
	\Supph ^{\geq i} ( F [n] \otimes^{\textbf{L}}_{\CO_X} B)
	& = \Supph ^{\geq i} (( F \otimes^{\textbf{L}}_{\CO_X} B)[n]) \\
	& \subset \Supph ^{\geq i} ( F \otimes^{\textbf{L}}_{\CO_X} B).\\
	\end{align*}
	 Now,  applying $n = 0$ case we get
	 \[\Supph ^{\geq i} ( F [n] \otimes^{\textbf{L}}_{\CO_X} B) \subset \Supph ^{\geq i}(B).\]

\end{proof}

 \begin{definition}
 \label{U phi}
 A \emph{filtration of supports of $X$} is a function $\phi$ from $\Z$ with values in the set of 
 subsets of $X$ such that  $\phi (i+1) \subset \phi (i) $ for each $i \in \Z$.  
  Let $\phi$ be a filtration of supports of $X$.  The subcategory of $derive X$ associated to 
  $\phi$ and denoted by $\cat U_{\phi}$
  is the full subcategory containing objects $E$ such that $ \Supp (H^i(E)) \subset \phi (i)$ for 
  each $i \in \Z$. 
 
\end{definition}

\begin{proposition}
\label{U}
	Let $\phi$ be a filtration of supports of $X$ and  $\cat U_{\phi}$ be the associated 
	subcategory of $\derive X$.  Then $\cat U_{\phi}$ is a cocomplete \tensor preaisle of $\derive X$.
\end{proposition}

\begin{proof}
	Since $\phi$ is a decreasing function $E \in \cat U_{\phi}$ if and only if $\Supph ^{\geq i} (E) 
	 \subset \phi(i)$ for each $i \in \Z$.  By Lemma \ref{support lemma 2} if $E \in \cat U_{\phi}$, 
	 then $E[1] \in \cat U_{\phi}$.  Using Lemma 
	 \ref{support lemma 1} we can observe that $\cat U_{\phi}$ is closed under extensions,  and 
	 Lemma \ref{support lemma 3} implies $\cat U_{\phi}$ is cocomplete.  Now to show 
	 $\cat U _{\phi}$ is a \tensor preaisle, by  Lemma \ref{tensor test}, it is
	 enough to show for any flat $\CO _X$-module $F$ and any
	$B \in \cat U _{\phi}$ we must have $F \otimes ^{L} _{\CO _X} B \in \cat U_{\phi}$.  This
	follows from $n = 0$ case of Lemma \ref{support lemma 4}.
\end{proof}

\begin{lemma}
\label{support lemma 5}
	Let $B$ $\in \derive X$, and $E$ be a perfect complex in $\mathbf{D}_{qc} ^{\leq 0}(X)$.  
	Then,  \[ \Supph^{\geq i} ( E \otimes^{\textbf{L}}_{\cat O_X} B) \subset \Supph^{\geq i} (B).\]
\end{lemma}

\begin{proof} 
	
 First, we observe that the brutal truncation of a perfect complex is perfect.  For each positive 
 integer $n$, we have the triangle coming from brutal 
truncation,
\[ \sigma^{\geq -(n -1)} E \rar \sigma ^{\geq -n} E \rar E^{-n}[n] \rar \sigma^{\geq -(n-1)} E[1].\]
The $\cat O _X$-module $E^{-n}$ is flat.  Tensoring with $B$ we get the triangle,
\[ \sigma^{\geq -(n -1)} E \otimes^{\textbf{L}}_{\CO_X} B  \rar \sigma ^{\geq -n} E 
 \otimes^{\textbf{L}}_{\CO_X} B\rar E^{-n}[n]  \otimes^{\textbf{L}}_{\CO_X} B \rar 
 \sigma^{\geq -(n-1)} E[1] \otimes^{\textbf{L}}_{\CO_X} B.\]

As $E$ is perfect and $X$ is quasi-compact there is an $N$ such that 
$ \sigma ^{\geq -N} E = E$.
In view of Lemma \ref{support lemma 1} and by induction on the length of $E$,  it is now 
enough to show for any flat $\CO _X$-module $F$ and positive integer $n$,
$\Supph ^{\geq i} ( F [n] \otimes^{\textbf{L}}_{\CO_X} B) \subset \Supph ^{\geq i}(B)$;
which has been shown in \ref{support lemma 4}.  This proves the claim.

\end{proof}

\subsection{Extending perfect complexes along a tensor preaisle}
\label{Ext}

Let $X$ be a quasi-compact quasi-separated scheme with an ample family of line bundles.  Let 
$U \subset X $ be a quasi-compact open subset and $Z$ be a closed subset of $X$ with 
$X \setminus Z$ quasi-compact.

The following proposition is an improvement of  \cite [Proposition 5.4.2]{TT90},  we remove 
both the Noetherian hypothesis and the boundedness assumption on $F$.  
Also we observe that the map can be extended along a \tensor preaisle.

\begin{proposition}
\label{extension of morphism}
	
	Let $E \in \perfect X$ and $F \in \derive X$ and $a : E|_U \rar F|_U$ be a map in $\derive U$.  
	Then, there exists $\prim E \in \perfect X$,  and maps $c : \prim E \rar E$ and
	$b : \prim E \rar F$ such that  $c|_U : \prim E|_U \rar E|_U$ is an isomorphism
	 and $ a \circ c|_U = b|_U$.   Moreover, 
	
	If \cat U is a \tensor preaisle and $E \in \cat U$, then $\prim E$ can be chosen such that  $
	\prim E \in \cat U$.

\end{proposition}

\begin{proof}
	First, we note that in the proof of  \cite [Proposition 5.4.2]{TT90} the Noetherian assumption 
	on $X$ and bounded below assumption on $F$ are being used to make sure the map 
	between $E|_U \rar F|_U$ is a strict map of complexes.  Once the map $E|_U \rar F|_U$ is 
	a strict map then by the second half of the proof of  \cite [Proposition 5.4.2]{TT90},
 	this improved version follows verbatim.
	
	To have a strict map of complexes it is enough if $F|_U$ is K-injective.  Let $j : U \rar X $ be 
	the open immersion.  The functor $j^*$ takes K-injectives to K-injectives as it
	has an exact left adjoint.  For any $F$ in \derive X,  we can replace $F$ by a K-injective 
	resolution.  
	Hence without of loss of generality, we can assume $F$ is K-injective and therefore 
	$j^*F$ is K-injective.
	
	From the proof of Proposition 5.4.2 \cite{TT90}, one can take 
	$\prim E =  K^+ \otimes^{\textbf{L}}_{\cat O_X} E$ 
	where $K^+$ is a perfect complex in $\mathbf{D}_{qc} ^{\leq 0}(X)$.  The last part of 
	the claim follows since $\cat U$ is a \tensor preaisle.  
	
\end{proof}

We use the shorthand notation \perfn X for \perfect X $\cap$  $\mathbf{D}_{qc} ^{\leq N}(X)$.   
We say a complex $A$ in \perfn U has an extension in \perfn X if there is an object $\prim A \in 
\perfn X$ such that $\prim A|_U \cong A$.  

\begin{lemma}
\label{extension lemma 1}
	 Let $A \rar B \rar C \rar A[1]$ be a distinguished triangle in \perfn U. If $A$ and 
	$B$ have extensions in \perfn X then $C$ has an extension  in \perfn X.
\end{lemma}

\begin{proof}
	Suppose we have $\prim A$ and $\prim B$ in $\perfn X$ such that 
	$\prim A|_U \cong A$ and $\prim B|_U \cong B$.  
	Note that  $\mathbf{D}_{qc} ^{\leq N}(X)$ is a \tensor preaisle.
	Using Proposition \ref{extension of morphism}, if required replacing $\prim A$, the map 
	$A\rar B$ can be extended to a map $\prim A \rar \prim B$ in $\perfn X$.  The cone of 
	the map $\prim A \rar \prim B$ is an extension of $C$ in $\perfn X$.
\end{proof}

The following lemma is a modification of \cite[Lemma 5.5.1]{TT90}.  To illustrate the 
difference we write the proof in a similar fashion as in \cite{TT90}.

\begin{lemma}
\label{extension key lemma}
	For any $F$ in \perfn U there exists a perfect complex $\bar{ F}$ in \perfn U 
	such that $F \oplus \bar F$ has an extension in \perfn X.
\end{lemma}

\begin{proof}
	Let $j : U \rar X$ be the open immersion.  Consider ${\bf{R}}j_*F$ on $X$.  This is 
	cohomologically bounded below with quasi-coherent cohomology.  The same is true for
	$\tau ^{\leq N} {\bf{R}}j_*F$,  so we can apply \cite[Corollary 2.3.3]{TT90}, which says, there
	is a directed system of strict perfect complexes $\{E_{\alpha}\}$ bounded above by $N$ 
	such 
	that  \[ \varinjlim   E_{\alpha} \cong \tau ^{\leq N} {\bf{R}}j_*F.\]
	
	Applying the exact functor $j^*$ we get,
	\begin{align*}
	\varinjlim j^*E_{\alpha} 
	& \cong j^*\tau ^{\leq N} {\bf{R}}j_*F\\
	& \cong \tau ^{\leq N} j^*{\bf{R}}j_*F\\
	& \cong  \tau ^{\leq N} F\\
	& \cong F.  \\
	\end{align*}
	
	As $F$ is perfect the isomorphism $F \rar \varinjlim j^*E_{\alpha}$ must factor 
	through 
	some $j^* E_{\alpha}$.  Since every monomorphism splits in a triangulated category,  there is 
	an object $\bar F$ such that $F \oplus \bar F \cong j^*E_{\alpha}$.  As $\perfn U$ is 
	a thick preaisle $\bar F \in \perfn U$ and $E_{\alpha}$ is an extension of 
	$F \oplus \bar F$ in $\perfn X$. 
\end{proof}

\begin{lemma} 
\label{extension unrestricted}
	For any $F$ in \perfn U the perfect complex $F \oplus F[1]$ has an extension in \perfn 
	X.
\end{lemma}

\begin{proof}
	By Lemma \ref{extension key lemma}, for a given $F$ we have $\bar F$ such that $F \oplus 
	\bar F$ has an extension in \perfn X.  Consider the direct sum of the following 
	distinguished triangles
	\[0 \rar F \rar F \rar 0[1],\]
	\[\bar F \rar \bar F \rar 0 \rar \bar F[1],  \]
	\[F \rar 0 \rar F[1] \rar F[1].\]
	
	By Lemma \ref{extension lemma 1},  $F \oplus F[1]$ has an extension in \perfn X.
\end{proof}

Next, we prove a version of the extension lemma with support conditions.  We denote the 
collection of perfect complexes $E \in \perfn X$ with $\Supph (E) \subset Z$ by 
$\mathrm{Perf} ^{\leq N} _{Z} (X)$. 

\begin{proposition}
\label{extension restricted}
	For any $F$ in $\mathrm{Perf} ^{\leq N} _{Z\cap U} (U)$ the perfect complex 
	$F \oplus F[1]$ has an extension in $\mathrm{Perf} ^{\leq N} _{Z} (X)$. 
\end{proposition}

\begin{proof}
	We denote  $U \cup (X \setminus Z)$ 
	by $W$.   Take the open immersion $k : U \rar W$.  Recall the functor $k_{!}$
	extending by zero is exact and a left adjoint of $k^*$.  We take $k_!F$,  which is a 
	perfect complex on $W$,  since locally it is perfect: $k_!F|_U = F$ and 
	$k_!F|_{(W \setminus Z)} \cong 0$.
	
	Now we apply the unrestricted extension Lemma \ref{extension unrestricted} to 
	$k_!F \oplus k_!F[1]$.  This shows $F \oplus F[1]$ has an extension with the 
	support condition.
\end{proof}

\subsection*{Without the divisorial condition}

Next, we proceed to remove the hypothesis of ample family of line bundles,  using 
homotopy pushout construction of \cite[3.20.4]{TT90}.

\begin{proposition}
\label{K + K}
	Let $X$ be a quasi-compact quasi-separated scheme.  Let $U \subset X $ be a quasi
	compact open subset and $Z$ be a closed subset of $X$ with $X \setminus Z$ quasi 
	compact.   For any $F$ in $\mathrm{Perf} ^{\leq N} _{Z\cap U} (U)$ there is a positive 
	integer $n$ such that the perfect complex 
	$\oplus _{i = 0} ^{n} (F[i]^{\oplus \binom {n} {i}})$ has an extension in 
	$\mathrm{Perf} ^{\leq N} _{Z} (X)$. 
\end{proposition}

\begin{proof}
	There exists a finite set $\{V_1 ,  \dots ,V_n\}$ of open affine subsets of $X$ such that 
	$X = U \cup V_1 \dots  \cup V_n$.  We prove the result recursively.  
	
	The case n = 1.  For $U \cup V_1$ and $F \in \mathrm{Perf} ^{\leq N} _{Z\cap U} (U)$.  
	Consider 
	the complex $F \oplus F[1]$.  Since $V_1$ is divisorial, by Proposition 
	\ref{extension restricted} the complex $(F \oplus F[1])|_{U \cap V_1}$ has an extension say
	$ F_{V_1}$ in $\mathrm{Perf} ^{\leq N} _{Z\cap V_1} (V_1)$.  Now we take the homotopy 
	pushout of $ F_{V_1}$ on $V_1$ and $F \oplus F[1]$ on $U$, by the methods of \cite[3.20.4]
	{TT90}, which will provide an extension, say $F_1$,  of $F \oplus F[1]$ in $
	\mathrm{Perf} ^{\leq N} _{Z \cap (U \cup V_1)} (U \cup V_1)$.
	
	Next for $U \cup V_1 \cup V_2$ we repeat the
	construction of step $n = 1$ by replacing $U$ by $U \cup V_1$, $V_1$ by $V_2$ and
	$F$ by $F_1$.  If we denote the extension of $F_1 \oplus F_1[1]$ constructed in 
	this manner by $F_2$,  then $F_2 |_U \cong F \oplus F[1]^{\oplus 2} \oplus F[2]$.  
	Repeating the construction $n$ times we obtain $F_n$, which is an extension of
	 $\oplus _i (F[i])^{\oplus \binom {n} {i}} $ in $\mathrm{Perf} ^{\leq N} _{Z} (X)$. 
	
\end{proof}

\begin{lemma}
\label{non zero map}
	Let $X$ be a quasi-compact quasi-separated scheme.  Let $U \subset X $ be a quasi 
	compact open subset and $Z$ be a closed subset of $X$ with $X \setminus Z$ quasi 
	compact. 
	Let $F \in \derive X$.  If $F \in (\mathrm{Perf} ^{\leq N} _{Z} (X))^{\perp}$ then $F|_U \in 
	(\mathrm{Perf} ^{\leq N} _{Z\cap U} (U))^{\perp}$.
\end{lemma}

\begin{proof}
	For simplicity, we denote $\mathrm{Perf} ^{\leq N} _{Z} (X)$ by $\cat S (X)$ and 
	$\mathrm{Perf} ^{\leq N} _{Z\cap U} (U)$ by $\cat S (U)$.  There exists a finite set $\{V_1 ,  
	\dots ,V_n\}$ of open affine subsets of $X$ such that $X = U \cup V_1 \dots  \cup V_n$.  We 
	prove the result by induction on the number $n$ of open affine subsets.
	
	For $n = 1$.  Let $X = U \cup V$ and $F \in \derive X$.  If $F|_U \notin (\cat S(U))^{\perp}$ 
	then there is a non zero map $E \rar F|_U$ for some $E \in \cat S(U)$.  Now we take the map
	$E \oplus E[1] \rar E  \rar F|_U$.
	
	Since $V$ is divisorial  by Proposition \ref{extension 
	restricted} $(E \oplus E[1])|_{U \cap V}$ has an extension say $\prim E \in \cat S (V)$.  By
	Proposition \ref{extension of morphism} if needed replacing $\prim E$ we can extend the
	map $(E \oplus E[1])|_{U \cap V} \rar F|_{U \cap V}$ to a map $\prim E \rar F|_V$.
	Now we take the homotopy pushout
	of the maps $E \oplus E[1] \rar F|_U$  and $\prim E \rar F|_V$ by the methods
	of \cite[3.20.4]{TT90}.  This shows $F \notin (\cat S(X))^{\perp}$.
	
	Now the induction step.  We have $X = U \cup V_1 \dots  \cup V_n$.  Let 
	$W =  U \cup V_1 \dots  \cup V_{n-1}$.  By induction hypothesis we have the result for $W$,
	that is,  if $F|_U \notin (\cat S(U))^{\perp} $ then $F|_W \notin (\cat S (W))^{\perp}$.   Using 
	the case $n = 1$ for $W$ and $V_n$ we can obtain $F \notin (\cat S(X))^{\perp}$.
\end{proof}

\subsection{Compact generation of  tensor preaisles}

\begin{definition}
\label{D Thomason subset}
	A subset $Z$ is a \emph{specialization closed} subset of $X$ if for each $x \in Z$ the 
	closure of the singleton set $\{x\}$ is contained in $Z$, that is, $\bar{\{x\}} \subset Z$.  Note 
	that a specialization closed subset is a union of closed subsets of $X$.  A subset $Y$ is 
	a \emph{Thomason} subset of $X$ if $Y = \bigcup _{\alpha} Y_{\alpha}$ is a union of 
	closed subsets $Y_{\alpha}$ such that $X \setminus Y_{\alpha}$ is quasi-compact. 
	Thomason subsets are specialization closed subsets but the converse need not be true
	in general.
	However,  if $X$ is Noetherian then the two notions coincide since every open subset
	of a Noetherian scheme is quasi-compact.  
	
	A \emph{Thomason filtration} of $X$ is a map 
	$\phi : \Z \rar 2^X $ such that $\phi (i)$ is a 
	Thomason subset of $X$ and $\phi (i) \supset \phi(i+1)$ for all $i \in \Z$. 
	
\end{definition}

Let $R$ be a Noetherian ring and $Z$ be a Thomason subset of $\spec R$.  Suppose $i$ is a 
fixed integer.  Consider the Thomason filtration $\psi$ which is defined as
	\begin{align*}
		 \psi (j) 
		& = Z \hspace{1cm}\text{if } j \leq i; \\
		& =  \emptyset \hspace{1cm} \text{if } j > i .\\
	\end{align*}
We denote the subcategory of $\mathbf{D}(R)$ associated with $\psi$ by $\cat U^i _Z$.
Consider the following set of compact objects of $\mathbf{D}(R)$,
\[  \cat K ^i _Z  = \{ K( a_1 , a_2 \dots a_n)[-i] \mid \{a_1 , a_2 , \dots a_n\} \subset R \text{ and } 
	V(\langle ( a_1 , a_2 \dots a_n) \rangle) \subset Z \},\]

where $K( a_1 , a_2 \dots a_n) $ denote the Koszul complex associated to the set
	 $\{a_1 , a_2 , \dots a_n\} \subset R.$

\begin{lemma}
\label{Koszul 1}
	The preaisle $\cat U ^i _Z$ is compactly generated by $\cat K ^i _Z$.
\end{lemma}

\begin{proof}
	\cite[Corollary 3.9]{AJS10}.
\end{proof}

The following result is well known and first appeared in \cite[Theorem 3.10]{AJS10}. 
Since it is a crucial step in achieving Theorem \ref{Compact generation},  we include a proof.

 \begin{proposition}
\label{compact affine}
	Let $R$ be a commutative Noetherian ring. 
	Let $\phi$ be a Thomason filtration of $\spec R$.
	The preaisle $\cat U _{\phi}$ is compactly generated by the following set
	\[  \cat K_{\phi}  = \bigcup_{i \in \Z} \cat K ^i _{\phi(i)}.\]
	
\end{proposition}

First we prove a lemma.  
\begin{lemma}
\label{Koszul 2}
	Let $\phi$ be a Thomason filtration of $\spec R$.  Consider the preaisles $\cat U 
	_{\phi}$ and $\cat U ^i _{\phi(i)}$.
	  Then we have,
	\[ \cat U _{\phi} ^{\perp} = \bigcap _{i \in \Z} (\cat U ^i _{\phi(i)})^{\perp}. \]
\end{lemma}

\begin{proof}
	From the definition we have $ \cat U ^i _{\phi(i)} \subset \cat U _{\phi}$ for each $i$ hence 
	 \[ \cat U _{\phi} ^{\perp} \subset \bigcap _{i \in \Z} (\cat U ^i _{\phi(i)})^{\perp}. \]
	 
	 For the reverse inclusion, we prove by contradiction.  Suppose $A \in 
	 \bigcap _{i \in \Z} (\cat U ^i _{\phi(i)})^{\perp}$ and $A \notin \cat U _{\phi} ^{\perp}$.  Then 
	 there is an object $B \in \cat U _{\phi}$ such that $\mathrm{Hom} ( B , A) \neq 0$.  This 
	 means there is a non-zero map $f:B \rar A$.  If $B$ is bounded above, then for some
	 $i$,  $B \in \cat U^i _{\phi (i)}$.  This is a contradiction as 
	 $A \in (\cat U^i _{\phi (i)})^{\perp}$.  
	 Now, suppose $B$ is not bounded above.  Consider the map $f_n : \tau^n B \rar \tau^n A \rar 
	 A$ where $\tau^n A \rar A$ is the natural inclusion.  The map $f$ is a filtered colimit of the
	 sequence of maps $\{f_n\}$,  since  $f \neq 0$, all $f_n$ can not be zero.   Therefore,  there is 
	 an $i$ such that $f_i : \tau ^i B \rar A$ is non-zero.  This is again a contradiction as
	  $A \in (\cat U^i _{\phi (i)})^{\perp}$ .
\end{proof}

\begin{proof}[Proof of Proposition \ref{compact affine}]
	Let $A \in \cat U _{\phi}$ and $A \in (\cat K _{\phi} )^{\perp}$.  Note that $A \in (\cat 
	K_{\phi}) ^{\perp}$ implies $A \in (\cat K ^i _{\phi(i)}) ^{\perp}$ for all $i$.  From lemma 
	\ref{Koszul 1} we have $A \in (\cat U ^i _{\phi(i)})^{\perp}$ for all $i$.  By lemma 
	\ref{Koszul 2} we have $A \in  \cat U _{\phi}^{\perp}$ this means $A \in \cat U _{\phi} 
	\bigcap \cat U _{\phi}^{\perp}$ hence $A \cong 0$.  This proves,  by Proposition 
	\ref{compact characterization}, $\cat U _{\phi}$ is compactly generated by the set 
	$\cat K _{\phi}$.
\end{proof}

\begin{theorem}
\label{Compact generation}
	
	Let $X$ be a Noetherian scheme.
	Let $\phi$ be a Thomason filtration of $X$.  The \tensor preaisle $\cat U_{\phi}$ of $\derive X$ is 
	compactly generated.  
\end{theorem}

\begin{proof}
	Consider the following essentially small set of compact objects on 
	$X$,
	
	\[ \cat S _{\phi} = \bigcup _{i \in \Z}  \mathrm{Perf} ^{\leq i} _{\phi(i)}(X). \]
	
	We will show $\cat U _{\phi} = 
	\langle {\cat S_{\phi}} \rangle ^{\leq 0} .$	 
	 By part (4) of Proposition \ref{compact characterization}  it is enough to show 
	if $A \in \cat U_\phi$ and $\Hom (S,  A) = 0$ for all $S \in \cat S_{\phi}$ then 
	$A \cong 0$.   Suppose $A \in \cat U_\phi$ and $A \in \cat S_{\phi} ^{\perp} = 
	\bigcap_i (\mathrm{Perf} ^{\leq i} _{\phi(i)}(X))^{\perp}$.  Let $U$ be an 
	open affine subset of $X$.  We denote the restriction of 
	$\phi$ on $U$ by $\phi |_U$.  For any $K \in 
	\cat K _{\phi |_U} = \bigcup _i \cat K ^i _{\phi |_U (i)}$ we have $\Hom ( K ,  A|_U) = 0$,  by 
	Lemma \ref{non zero map}.
	 Now by Proposition \ref{compact affine}
	we get $A|_U \cong 0$.  Since $U$ is arbitrary we get $A \cong 0$.
\end{proof}

\section{The classification theorem}
\label{section 4}

\subsection{Graded cohomological support of subcategories}

\begin{definition}
\label{Phi U}
Let $\cat U \subset \derive{X}$ be a subcategory.   The
 \emph{graded support of $\cat U$} is a function $\gsupp{\cat U}$, defined as
 \begin{center}
 	$ \gsupp{\cat U}(i) = \{ x \in X \mid \exists$  $E \in \cat U$ such that $x \in \Supp (H^i(E)) \}
 	$.
 \end{center}
 
\end{definition}

\begin{lemma}
\label{support}
	Let $\cat U = \taisle S$ be a \tensor aisle of \derive X generated by a set of compact objects 
	\cat S.  Then
	\[ \gsupp {\cat U}(i) = \bigcup_{S \in \cat  S} \Supph^{\geq i} (S).\]
	In other words,  the graded support of \cat U can be computed from a set of compact
	generators.
\end{lemma}

\begin{proof}

	For a preaisle \cat U,  since it is closed under positive shifts, we have
	\[ \gsupp {\cat U}(i) = \bigcup_{E \in \cat U} \Supp (H^i(E)) = \bigcup_{E \in \cat U}
	\Supph^{\geq i} (E).\] Clearly,
	\[  \bigcup_{S \in \cat  S} \Supph^{\geq i} (S) \subset \bigcup_{E \in \cat U}\Supph^{\geq i} 
	(E).\]

	By Proposition \ref{keller tensor}(2) and the support lemmas \ref{support lemma 1}, 
	\ref{support lemma 2},  \ref{support lemma 3}, and \ref{support lemma 5} for
	any $E \in \cat U$ we have $\Supph ^{\geq i} (E) \subset  \bigcup_{S \in \cat  S} 
	\Supph^{\geq i} (S)$.  This proves our claim.
	
\end{proof}

\begin{lemma}
\label{Thomason subset}
	Let $X$ be a Noetherian scheme and \cat U be a compactly generated $\otimes$-aisle of 
	\derive X.  Then \gsupp {\cat U} is a Thomason filtration of $X$.
\end{lemma}

\begin{proof}
	As $X$ is quasi-compact a perfect complex $S$ on $X$ is bounded.  Also, the
	cohomology sheaves of $S$ are finite type $\CO _X$-modules.
	Therefore $\Supph ^{\geq i} (S) $ is a closed subset of $X$; as it is a finite union of 
	closed subsets.  By Lemma \ref{support} the set \gsupp {\cat U}$(i)$ is a Thomason 
	subset.  Since \cat U is closed under positive shifts we have $\gsupp {\cat U}(i+1) \subset 
	\gsupp {\cat U}(i)$ for each $i$,  hence $\gsupp {\cat U}$ is a Thomason 
	filtration.
\end{proof}

\subsection{The classification theorem for Noetherian schemes}

Let $X$ be a Noetherian scheme and $U$ be an open affine subset of $X$ and $j : U \rar X$ 
denote the open immersion.  Let $\cat U$ be a \tensor preaisle of $\derive X$.  We define 
$\cat U|_U \colonequals \langle j^* \cat U \rangle ^{\leq 0}$ - the restriction of 
$\cat U$ to the open affine subset $U$.  Suppose $(\cat U , \cat V[1])$ is a tensor t-structure on 
\derive X.  We define $\cat V|_U \colonequals (\cat U|_U)^{\perp}$.

\begin{remark}
	In the earlier manuscript, there was an error in the proof of Lemma \ref{j!}. The proof was 	
	based on an erroneous lemma, which was used to give a shorter proof of Lemma 
	\ref{j!}.  The error was first pointed out to us by Alexander Clark.  Later Suresh 
	Nayak also pointed out the same error.  We sincerely thank both of them for bringing our 
	attention to it.  In this version,  we give a correct proof of the lemma. 
\end{remark}

\begin{lemma}
\label{j!}
	Let $\cat U$ be a \tensor preaisle of \derive X.   If $F \in (\cat U)^{\perp}$ then $j^*F \in 
	(\cat U|_U)^{\perp}$.  
\end{lemma}

\begin{proof}

	First we prove the following claim:  If $A \notin (\cat U|_U)^{\perp}$ then 
	there is an element $E \in \cat U$ such that $\Hom (j^*E,A) \neq 0$.  Take the collection 
	$\{ A[-n] \mid n \geq 0 \}$ and denote it by $\cat A$.  Note that 
	${}^\perp \cat A$ is a preaisle.  If for all $E \in \cat U$,  $\Hom (j^*E, A) = 0$ then the preaisle
	${}^\perp \cat A$ contains all $j^*E$ for $E \in \cat U$.  Hence,  ${}^\perp \cat A$ contains
	$\cat U|_U$.  This is a contradiction since $A \notin (\cat U|_U)^{\perp}$.  Now we proceed
	to prove the lemma.

	Suppose there is an $F \in (\cat U)^{\perp}$ such that $j^*F \notin
	(\cat U|_U)^{\perp}$ then by our claim there is an $E \in \cat U$ such that  
	$\Hom (j^*E, j^*F) \neq 0$.  By tensor-hom adjunction we get 
	
	\begin{align*}
	\Hom (j^*E, j^*F)
	& = \Hom (j^* \CO _X ,  \mathbf{R}\scr{H}om_{U}(j^*E,  j^*F)) \\
	& = \Hom (j^* \CO _X ,  j^* \mathbf{R}\scr{H}om_{X}(E,  F)).\\
	\end{align*}
	
	Now if we write $\prim F = \mathbf{R}\scr{H}om_{X}(E,  F)$ then  $\Hom (j^*E, j^*F) \neq 0$
	implies the group $\Hom (j^* \CO_X ,  j^* \prim F) \neq 0$.  Applying Lemma 
	\ref{non zero map} to the preaisle $\mathrm{Perf} ^{\leq N} _{Z} (X)$ where $N=0$ and
	$Z=X$,  we get that there is a perfect complex $K \in \mathrm{Perf} ^{\leq 0} (X)$ such that
	$\Hom (K, \prim F) \neq 0$.  Again by tensor-hom adjunction 
	\begin{align*}
	\Hom (K,  \prim F)
	& = \Hom (K ,  \mathbf{R}\scr{H}om_{X}(E,  F)) \\
	& = \Hom (K \otimes E,  F) \neq 0,\\
	\end{align*}
	since $\cat U$ is a tensor preaisle, $K \otimes E \in \cat U$ and we get a contradiction.

\end{proof}

\begin{lemma}
\label{localizes to open}
	Let $(\cat U , \cat V[1])$ be a tensor t-structure on \derive X.  Then 
	$(\cat U|_U,  \cat V|_U[1])$ 
	is 
	a t-structure on \derive U.  In particular,  for any $A \in \derive X$ the triangle 
	 \[ j^*\tau_{\cat U} ^{\leq} A \rar j^*A \rar j^*\tau _{\cat U} ^> A \rar  j^*\tau_{\cat U} 
	^{\leq} A[1],\] is a t-decomposition triangle of $j^*A$ in \derive U.
\end{lemma}

\begin{proof}
	To show $(\cat U |_U , \cat V |_U[1])$ is a t-structure 
	it is enough to get a t-decomposition triangle for each $B \in \derive U$.
	For any $B \in \derive U$,  take $j_*B$ and its t-decomposition
	triangle for $(\cat U,  \cat V[1])$ 
	\[ \tau ^{\leq} _{\cat U} j_*B \rar j_*B \rar \tau ^{>} _{\cat U} j_*B \rar \tau ^{\leq} _{\cat U} 
	j_*B[1]. \]
	
	Applying $j^*$ we get a triangle in $\derive U$.   Since $j^*j_*B \cong B$ we get
	\[ j^*\tau ^{\leq} _{\cat U} j_*B \rar B \rar j^*\tau ^{>} _{\cat U} j_*B \rar j^* \tau ^{\leq} 
	_{\cat U} j_*B[1]. \]
	
	By Lemma \ref{j!} $j^*\tau ^{>} _{\cat U} j_*B \in \cat V|_U$ and by definition of $\cat U|_U$
	we have $j^*\tau ^{\leq} _{\cat U} j_*B\in \cat U|_U$,  hence the above triangle gives a 
	t-decomposition of $B$ for  $(\cat U |_U , \cat V |_U[1])$.  
	
	Now for any $A \in \derive X$ applying $j^*$ and the argument as above, we get the
	second half of the claim.
\end{proof}

\begin{lemma}
	If $\cat U$ is a compactly generated \tensor aisle of \derive X,  then the restriction $\cat 
	U|_U$ is a compactly generated \tensor aisle of \derive U.
\end{lemma}

\begin{proof}
	The restriction of a perfect complex is perfect hence by Lemma \ref{key1} and Proposition
	\ref{keller tensor}(1) the claim follows.
\end{proof}

\begin{lemma}
\label{affine lemma}
	Let $R$ be a Noetherian ring.  Let $A$ and $B$ be two perfect complexes on $R$.  If 
	$\Supph^{\geq i} (A) \subset \Supph^{\geq i}  (B)$  for each $i$,  then $A \in \langle B \rangle 
	^{\leq 0}$.
\end{lemma}

\begin{proof}
	\cite[Proposition 5.1]{Hrbek}
\end{proof}

\begin{lemma}
\label{Lemma 2}
	Let $A$ and $B$ be 
	two perfect complexes on $X$.  If 
	$\Supph^{\geq i} (A) \subset \Supph^{\geq i}  (B)$  for each $i$,  then $A \in \langle B \rangle 
	^{\leq 0} _{\otimes}$.
\end{lemma}

\begin{proof}
	Consider the t-decomposition of $A$ with respect to the aisle $ \langle B \rangle ^{\leq 0} 
	_{\otimes}$
	\[ \tau_B ^{\leq} A \rar A \rar \tau_B ^> A \rar \tau_B ^{\leq} A [1].\]
	
	Let $U$ be an open affine subset of $X$, and $j: U \rar X$ be the open immersion.  Applying
	$j^*$ to the above triangle we get 
	\[ j^*\tau_B ^{\leq} A \rar j^*A \rar j^*\tau_B ^> A \rar j^*\tau_B^{\leq} A [1].\]
	
	Now using Lemma \ref{localizes to open} and the Lemma \ref{affine lemma} 
	we conclude the map $j^*\tau_B ^{\leq} A \rar j^*A$ is an isomorphism.   As the map 
	$ \tau_B ^{\leq} A 
	\rar A$ is locally an isomorphism on every open affine subset of $X$,  it is an  isomorphism.  
	This proves $A \in \langle B \rangle ^{\leq 0} _\otimes$.
	
\end{proof}

\begin{lemma}
\label{Lemma 1}
	Let $Z$ be a closed subset of $X$.  Then there is a perfect complex 
	$E \in \mathrm{Perf} ^{\leq 0} _Z (X)$ such that $\Supp (H^0(E)) = Z$.  
\end{lemma}

\begin{proof}
	Consider an open affine cover of $X$  say
	$\{ U_{\alpha} = \mathrm{Spec} R_{\alpha} \}$.  We take a Koszul complex $K_{\alpha}$
	 on $R_{\alpha}$ such that $\Supph(K_{\alpha}) = Z \cap U_{\alpha}$.  By Proposition 
	 \ref{K + K} there is an extension of $\oplus _i (K_{\alpha}[i])^{\oplus \binom {n} {i}}$;  
	 say $E_{\alpha}$.  Now take $E = \bigoplus_{\alpha} E_{\alpha}$,  it is easy to check 
	$\Supp (H^0(E)) = Z$, and by our construction $E \in \mathrm{Perf} ^{\leq 0} _Z (X)$.
\end{proof}

We denote the collection of Thomason filtrations of $X$ by \fil X and the collection 
of compactly generated \tensor aisles of $\derive X$ by \caisle X.

\begin{theorem}
\label{Theorem 1}
	Let $X$ be a Noetherian scheme.  There is a one-to-one 
	correspondence between \fil X  and  \caisle X.   More precisely,  the maps
	
	\[ \Phi: \caisle X \rar \fil X,\]
	\[ \cat U \mapsto \gsupp {\cat U}, \hspace{5mm} (\ref{Phi U}),\]

	\[ \Psi : \fil X \to \caisle X,\]
	 \[\phi \mapsto \cat U _{\phi},\hspace{5mm} (\ref{U phi}),\] 
	are inverse of each other.

\end{theorem}

\begin{proof}
	The maps are well defined by Lemma \ref{Thomason subset} and Theorem 
	\ref{Compact generation}.
	First, we show $\Phi \circ \Psi =$ Id. 
 	Given $\phi \in \fil X$ we have $\phi (i) = \bigcup _{\lambda} Z_{\lambda}$ where 
 	$Z_{\lambda}$'s are closed subsets of $X$.  For a fixed $Z_\lambda$ by Lemma 
 	\ref{Lemma 1}
 	we have a perfect complex $E_\lambda$ and $E_\lambda [i] \in \Psi (\phi)$,  this
 	proves the claim.
 	
 	To show $\Psi \circ \Phi$ = Id,  we first observe $\cat U \subset \Psi \circ \Phi(\cat U)$.
 	For the reverse inclusion, we will show that any compact object $A \in \Psi \circ \Phi(\cat U)$
 	is in $\cat  U$.  By Lemma \ref{support} we have compact objects $\{B_\alpha\} \subset
 	 \cat U$ such 
 	that $\Supph^{\geq i} (A) \subset \bigcup_\alpha \Supph^{\geq i} ( B_\alpha)$.  
 	
 	Since $A$ is perfect, $\Supph^{\geq i} (A)$ is a closed subset of $X$.  As $X$ is Noetherian
 	there are finitely many irreducible components of $\Supph^{\geq i} (A)$ say 
 	$Z_1$,  $Z_2$,  $\cdots, Z_n$.  Now, we take $B_i$ such that the generating point of
 	$Z_i$ lies in $\Supph^{\geq i} ( B_i)$.  Therefore,
 	we have a finite collection $\{B_i\}^n_{i = 1}$ of $\{B_\alpha\}$ such that 
	$\Supph^{\geq i} (A) \subset \bigcup_{i = 1} ^n \Supph^{\geq i} ( B_i)$.  So 
	$\Supph^{\geq i} (A) \subset \Supph^{\geq i} (\oplus_{i =1} ^n B_i)$, since
	 $\oplus_{i =1} ^n B_i \in \cat U$,  by
	Lemma \ref{Lemma 2} we conclude $A \in \cat U$.  This proves the theorem. 
 	
\end{proof}

\begin{remark}
\label{remark 1}
	From \cite[Theorem 4.10]{SP16}, it follows that, for a separated Noetherian scheme, there is 
	a natural bijection between the set of thick preaisles of \perfect X and the set of 
	compactly generated t-structures on $\mathbf{D}(\mathrm{Qcoh}(X))$.  By
	Proposition \ref{keller tensor} and its proof it can be easily deduced that the bijection 
	of \cite{SP16} restricts to a bijection 
	between the set of thick \tensor preaisles of \perfect X and the set of compactly generated 
	tensor t-structures on $\mathbf{D}(\mathrm{Qcoh}(X))$.  
\end{remark}

\begin{remark}
\label{remark 2}
	Theorem \ref{Theorem 1} together with Remark \ref{remark 1} says there is a bijective 
	correspondence between the set of thick \tensor preaisles of \perfect X and the set 
	of Thomason filtrations of $X$.  Therefore,  Theorem \ref{Theorem 1} can be 
	thought of as a generalization of \cite[Theorem 3.15]{Thomason} to thick \tensor preaisle,  at 
	least for separated Noetherian schemes.  
\end{remark}

\section{Tensor telescope conjecture for t-structures}

\label{section 5}

Let \cat G be a Grothendieck abelian category.  A subcategory \cat V of $\mathbf{D}(\cat G)$ is 
\emph{closed under homotopy colimits} if for any directed system $\{A_i\}$ in 
$\mathbf{C}(\cat G)$ with all $A_i$ in \cat V,  the colimit of the directed system $\{A_i\}$ in 
$\mathbf{C}(\cat G)$  belongs to \cat V;  see \cite[Fact 2.1]{HN21},  for the definition in the 
general setting of derivators see \cite[A.1]{HN21}.  A t-structure $(\cat U,  \cat V[1])$ on 
$\mathbf{D}(\cat G)$ is \emph{homotopically smashing} if the coaisle \cat V is closed under 
homotopy colimits.  Recall that we say a t-structure is smashing if the coaisle is closed under 
coproducts.

\begin{proposition}
	Every homotopically smashing t-structure on $\mathbf{D}(\cat G)$ is smashing.  
\end{proposition}

\begin{proof}
	\cite[Proposition 7.2]{SSV21}.
\end{proof}

\begin{remark}
\label{smashing ideal}
	Smashing t-structures in general are not homotopically smashing see \cite[Example 8.2]
	{SSV21}.  However,  in the case of stable t-structures these two notions coincide; see
	\cite[A.5]{HN21}.  
\end{remark}

\begin{proposition}
\label{smashing}
	Every compactly generated t-structure on $\mathbf{D}(\cat G)$ is homotopically smashing.  
\end{proposition}

\begin{proof}
	\cite[Proposition 7.2]{SSV21}.
\end{proof}

The telescope conjecture for t-structures asks if every homotopically smashing t-structure on 
$\mathbf{D}(\cat G)$ is 
compactly generated,  that is,  if the converse of Proposition \ref{smashing} is true.  When 
\cat G is $\mathrm{Mod}$-$R$ for $R$ Noetherian ring, Hrbek and Nakamura have proved the 
following:

 \begin{theorem}[\text{\cite[Theorem 1.1]{HN21}}]
 \label{telescope affine}
 	Any homotopically smashing t-structure on $\mathbf{D} (R)$ is 
 	compactly generated.  
 \end{theorem}
 
For a separated Noetherian scheme
$X$ and the derived category of quasi-coherent sheaves $\mathbf{D}(\mathrm{Qcoh}(X))$, 
which is equivalent to $\derive X$, we prove 
the following:

\begin{theorem}
\label{telescope conjecture}
	Any homotopically smashing tensor t-structure on $\mathbf{D}(\mathrm{Qcoh}(X))$ is 
	compactly generated.  
\end{theorem}

First, we prove some lemmas.  In this section $(\cat U, \cat V[1])$ will always mean a 
homotopically smashing tensor t-structure on $\mathbf{D}(\mathrm{Qcoh}(X))$ and
$U$ is an open affine subset of $X$.  

\begin{lemma}
\label{homotopically restricts}
	The restriction of $(\cat U,  \cat V[1])$ to $U$, that is, $(\cat U |_U , \cat V |_U[1])$ is a 
	homotopically smashing t-structure on $\mathbf{D}(\mathrm{Qcoh}(U))$.  
\end{lemma}

\begin{proof}
	Recall $\cat U|_U = \langle j^* \cat U \rangle ^{\leq 0}$ and 
	$\cat V|_U = (\cat U|_U)^{\perp}$.
	By Lemma \ref{localizes to open} $(\cat U |_U , \cat V |_U[1])$ is a t-structure on 
	$\mathbf{D}(\mathrm{Qcoh}(U))$. Now we will show $(\cat U |_U , \cat V |_U[1])$ is 
	homotopically smashing.  Let $\{ A_i\}$ be a directed system in $\cat V|_U$.  By adjunction 
	isomorphism, if $A_i \in \cat V|_U$ then $j_*A_i \in \cat V$.  Consider the directed system 
	$\{ j_*A_i \}$ in \cat V,  since
	$\cat V$ is closed under homotopy colimits we have $\varinjlim j_*A_i \in \cat V$.   Since
	$j^*$ is an exact functor and $j^*j_*A_i \cong A_i$ the colimit of the system $\{A_i\}$ is
	$j^*(\varinjlim j_*A_i)$ which by Lemma \ref{j!} is in 
	$\cat V|_U$. 
	
\end{proof}

\begin{lemma}
\label{Theorem 2: Lemma 2}
	If the graded support of \cat U is $\phi$, then the graded support of $\cat U |_U$ 
	is $\phi |_U$.  
\end{lemma}

\begin{proof}
	Consider $\cat U _{(\phi |_U)}$ the associated subcategory of the filtration $\phi|_U$ in 
	$\mathbf{D}(\mathrm{Qcoh}(U))$,  it is a cocomplete preaisle by Proposition \ref{U}.
	We have $j^*A \in \cat U _{(\phi |_U)}$ for any $A \in \cat U$ therefore 
	$\cat U|_U = \langle j^* \cat U \rangle ^{\leq 0} \subset \cat U _{(\phi |_U)}$.  So the graded 
	support of $\cat U|_U$ is contained in $\phi |_U$.  Next,  given $x \in \phi|_U(i)$,  we have 
	$x \in \phi (i)$ which means there is an object $E$ in \cat U such that $x \in \Supp (H^i(E))$.  
	As $j^*E \in \cat U|_U$ we get the graded support of $\cat U |_U$ is $\phi|_U$. 
\end{proof}

\begin{lemma}
\label{Theorem 2: Lemma 3}
	If the graded support of \cat U is $\phi$, then 
	$\cat U|_U = \cat U _{(\phi |_U)}$.
\end{lemma}

\begin{proof}
	By Lemma \ref{homotopically restricts} we know $\cat U |_U$ is homotopically 
	smashing.  As $U$ is affine, Theorem \ref{telescope affine} implies $\cat U |_U$ is compactly 
	generated.  By Lemma \ref{Theorem 2: Lemma 2} the graded support of $\cat U |_U$ is $
	\phi|_U
	$.  By the classification of compactly generated t-structures, we conclude
	 $\cat U|_U = \cat U _{(\phi |_U)}$.
\end{proof}
 
 \begin{lemma}
 \label{Theorem 2: Lemma 4}
	If the graded support of \cat U is $\phi$ then $\phi$ is a Thomason filtration. 
\end{lemma}

\begin{proof}
	By Lemma \ref{Theorem 2: Lemma 3} we have $\cat U|_U = \cat U _{(\phi |_U)}$ and
	$\cat U|_U$ is compactly generated.  Since the graded support of a compactly generated
	aisle is a Thomason filtration we get $\phi |_U$ is a Thomason 
	filtration.  In the Noetherian case, Thomason subsets and specialization 
	closed subsets are the same.  Specialization closed is a local property therefore the graded 
	support of \cat U is a Thomason filtration.

\end{proof}
 
 \begin{proof}[Proof of Theorem \ref{telescope conjecture}]
 We denote the graded support of \cat U by $\phi$ and consider $\cat U _{\phi}$ the 
 subcategory associated with $\phi$.   Clearly,  $\cat U \subset \cat U _{\phi}$.
 
 For $A \in \cat U_{\phi}$, consider the following t-decomposition triangle for \cat U
 \[ \tau^{\leq} _{\cat U} A \rar A \rar \tau^{>} _{\cat U} A \rar \tau^{\leq} _{\cat U} A [1]. \] 
 
By Lemma \ref{localizes to open} restricting to $U$ gives a t-decomposition triangle for 
$\cat U|_U$ \[ j^*\tau^{\leq} _{\cat U} A \rar j^*A \rar j^*\tau^{>} _{\cat U} A  \rar
 j^*\tau^{\leq} _{\cat U} A [1]. \] 
 
 Since $j^*A \in \cat U _{(\phi |_U)}$ and $\cat U|_U = \cat U _{(\phi |_U)}$ the map
 $ j^*\tau^{\leq} _{\cat U} A \rar j^*A$ is an isomorphism.  As it is true for each open affine 
 subset of $X$,
 we get $\tau^{\leq} _{\cat U} A \rar A $ is an isomorphism.  Thus 
 $A \in \cat U$ and this proves $\cat U = \cat U_{\phi}$.  By Lemma \ref{Theorem 2: Lemma 4}, 
 $\phi$ is a Thomason filtration. Now by Theorem \ref{Theorem 1}, the 
 subcategory $\cat U _{\phi}$ is a compactly generated \tensor aisle.  Therefore $\cat U$ is 
 compactly generated. 
\end{proof}

\begin{remark}
\label{remark 3}
	By Remark \ref{smashing ideal}  and Proposition \ref{tensor ideal}, the aisle of a 
	homotopically smashing stable tensor t-structure on $\mathbf{D}(\mathrm{Qcoh}(X))$ is a 
	smashing \tensor ideal of $\mathbf{D}(\mathrm{Qcoh}(X))$.  Therefore, Theorem 
	\ref{telescope conjecture} provides another proof of the tensor telescope conjecture for 
	separated Noetherian schemes.
\end{remark}

\section*{Acknowledgement}

We are grateful for the excellent work environment and the assistance of the support 
staff of HRI,  Prayagraj.  The second author is supported in part by the INFOSYS scholarship. 
We thank both Alexander Clark and Suresh Nayak for pointing out an error in a 
lemma which was used to prove a crucial step in the earlier manuscript.  We are thankful to 
Kapil Hari Paranjape and Michal Hrbek for reading the manuscript and suggesting 
improvements.  The authors also thank the anonymous referee for carefully going through the 
manuscript and providing valuable comments and suggestions.  
 
\bibliographystyle{alpha}
\bibliography{Project1}

 \end{document}